\documentclass[a4paper,12pt]{article}

\usepackage[utf8]{inputenc}
\usepackage[T1]{fontenc}
\usepackage{amsthm,amsmath,amssymb,amsfonts}
\usepackage{graphicx}
\usepackage[usenames,dvipsnames,svgnames,table]{xcolor}
\usepackage[colorlinks=true, pdfstartview=FitV,linkcolor=ForestGreen,citecolor=ForestGreen, urlcolor=blue]{hyperref}
\usepackage{a4wide}
\usepackage{wrapfig}

\newcommand{\dd}{\, \mathrm{d}}
\newcommand{\dx}{\, \mathrm{d}x}

\newcommand{\dv}{\, \mathrm{d}v}
\newcommand{\dz}{\, \mathrm{d}z}

\newcommand{\dt}{\, \mathrm{d}t}

\newcommand{\R}{\mathbb{R}}
\newcommand{\D}{\mathcal{D}}

\DeclareMathOperator{\osc}{osc}
\DeclareMathOperator{\esssup}{ess-sup}

\newcommand{\eps}{\epsilon}
\newcommand{\un}{\mathbb{I}}
\newcommand{\Qext}{Q_{\mathrm{ext}}}
\newcommand{\Qbext}{\mathcal{Q}_{\mathrm{ext}}}
\newcommand{\Qtext}{\tilde{Q}_{\mathrm{ext}}}
\newcommand{\Qzero}{Q_{\mathrm{zero}}}
\newcommand{\Qpos}{Q_{\mathrm{pos}}}

\newcommand{\Qint}{Q_{\mathrm{int}}}

\newcommand{\LK}{\mathcal{L}_K}

\newcommand{\Cwhi}{C_{\text{w.h.i.}}}

\newtheorem{definition}{Definition}[section]
\newtheorem{theorem}{Theorem}[section]
\newtheorem{proposition}{Proposition}[section]

\newtheorem{lemma}{Lemma}[section]

\theoremstyle{remark}
\newtheorem{remark}{Remark}

\title{\bf Log-transform and the weak Harnack inequality for kinetic
  Fokker-Planck equations} \author{Jessica Guerand \& Cyril Imbert}
\begin{document}
\maketitle

\begin{abstract}
  This article deals with kinetic Fokker-Planck equations with
  essentially bounded coefficients. A weak Harnack inequality for
  non-negative super-solutions is derived by considering their
  Log-transform and following S.~N.~Kruzhkov (1963).  Such a result
  rests on a new weak Poincar\'e inequality sharing similarities with
  the one introduced by W.~Wang and L.~Zhang in a series of works
  about ultraparabolic equations (2009, 2011, 2017). This functional
  inequality is combined with a classical covering argument recently
  adapted by L.~Silvestre and the second author (2020) to kinetic
  equations.
\end{abstract}

\section{Introduction}

This paper is concerned with local properties of solutions of linear
kinetic equations of Fokker-Planck type in some cylindrical domain $Q^0$
\begin{equation}
\label{e:main}
(\partial_t  + v \cdot \nabla_x) f = \nabla_v \cdot ( A \nabla_v f) + B \cdot \nabla_v f + S
\end{equation}
 assuming that the diffusion matrix
$A$ is uniformly elliptic and $B$ and $S$ are essentially bounded:
there exist $\lambda,\Lambda>0$ such that for almost every
$(t,x,v) \in Q^0$,
\begin{equation}
  \label{e:ellipticity}
  \begin{cases}   \text{eigenvalues of } A(t,x,v) = A^T(t,x,v) \text{ lie in } [\lambda,\Lambda], \\
\text{the vector field $B$ satisfies: } |B(t,x,v)| \le \Lambda.
\end{cases}
\end{equation}
In particular, coefficients do not enjoy further regularity such as
continuity, vanishing mean oscillation \textit{etc}. For this reason,
coefficients are said to be \emph{rough}.

\subsection{Main result}

We classically reduce the local study to the case where $Q^0$ is at
unit scale. For some reasons we expose below, $Q^0$ takes the form
$(-1,0]\times B_{R_0} \times B_{R_0}$ for some large constant $R_0$
only depending on dimension and the ellipticity constants
$\lambda,\Lambda$ in \eqref{e:ellipticity}.

Before stating our main result, we give the definition of (weak)
super-solutions in a \emph{cylindrical open set} set $\Omega$, that is
to say an open set of the form $I \times B^x \times B^v$.  A function
$f \colon \Omega \to \R$ is a weak super-solution of \eqref{e:main} in
$\Omega$ if
$f \in L^{\infty} (I, L^2 (B^x \times B^v)) \cap L^2 (I \times B^x,
  H^1(B^v))$ and
  $(\partial_t + v \cdot \nabla_x) f \in L^2(I \times B^x,
  H^{-1}(B^v))$ and for all non-negative $\varphi \in \D (\Omega)$,
  \[ 
- \int_{Q^0} f  (\partial_t + v \cdot \nabla_x) \varphi \dz \ge 
- \int_{Q^0}  A \nabla_v f \cdot  \nabla_v \varphi  \dz + \int_{Q^0} (B \cdot \nabla_v f+S) \varphi  \dz.
\]
 
\begin{theorem}[weak Harnack inequality]\label{t:weak-harnack}
  Let $Q^0 = (-1,0] \times B_{R_0} \times B_{R_0}$ and $A,B$ satisfy
  \eqref{e:ellipticity} and $S$ be essentially bounded in $Q^0$.  Let
  $f$ be a non-negative super-solution of \eqref{e:main} in
some cylindrical open set $\Omega \supset Q^0$. Then
  \[ \left( \int_{Q_-} f^p (z) \dz\right)^{\frac1p} \le C \left(
      \inf_{Q_+}f + \|S\|_{L^\infty (Q_+)} \right)\] where
  $Q_+ = (-\omega^2,0] \times B_{\omega^3} \times B_\omega$ and
  $Q_- = (-1, -1+\omega^2] \times B_{\omega^3} \times B_\omega$; the
  constants $C$, $p$, $\omega$ and $R_0$ only depend on the dimension and the ellipticity constants $\lambda, \Lambda$.
\end{theorem}
\begin{remark}
  Combined with the fact that non-negative sub-solutions are locally
  bounded \cite{pp}, the weak Harnack inequality implies the Harnack
  inequality proved in \cite{gimv}, see Theorem \ref{t:harnack}
\end{remark}
\begin{remark}
  The proof of Theorem~\ref{t:weak-harnack} is constructive. As a consequence,
  it provides a constructive proof of the Harnack inequality from \cite{gimv}.  
\end{remark}
\begin{remark}
  The (weak) Harnack inequality  implies  H\"older regularity of weak
  solutions.
\end{remark}
\begin{remark}
  Such a weak Harnack inequality can be generalized to the
  ultraparabolic equations with rough coefficients considered for
  instance in \cite{pp,wz09,wz11,wz17}. 
\end{remark}
\begin{remark}
  This estimate can be scaled and stated in arbitrary cylinders thanks
  to Galilean and scaling invariances of the class of equations of the
  form~\eqref{e:main}. These invariances are recalled at the end of
  the introduction.
\end{remark}
\begin{remark}
  As in \cite{gimv,is}, the radius $\omega$ is small enough so that
  when ``stacking cylinders'' over a small initial one contained in
  $Q_-$, the cylinder $Q_+$ is captured, see Lemma~\ref{l:cylinder}.
  As far as $R_0$ is concerned, it is large enough so that it is
  possible to apply the expansion of positivity lemma (see
  Lemma~\ref{l:pop}) can be applied to every stacked cylinder. 
\end{remark}

\subsection{Historical background and motivations}

The weak Harnack inequality from our main theorem and the techniques
we develop to establish it are deeply rooted in the large literature
about elliptic and parabolic regularity, both in divergence and
non-divergence form.

\paragraph{De Giorgi's theorem and Harnack inequality.}

E. De Giorgi proved that solutions of elliptic equations in divergence
form with rough coefficients are locally H\"older continuous
\cite{DG56,DG57}. This regularity result for linear equations allowed
him to solve Hilbert's 19\textsuperscript{th} problem by proving the
regularity of a non-linear elliptic equation. The case of parabolic
equations was addressed by J.~Nash in \cite{nash}. Then J.~Moser
\cite{zbMATH03179707,moser} showed that a Harnack inequality can be
derived for non-negative solutions of elliptic and parabolic equations
with rough coefficients by considering the logarithm of positive
solutions. The proof of E. De Giorgi applies not only to solutions of
elliptic equations but also to functions in what is now known as the
elliptic De Giorgi class. Parabolic De Giorgi classes were then
introduced in particular in \cite{L1}.

\paragraph{The log-transform.} While the proof of the continuity of
solutions for parabolic equations by J. Nash \cite{nash} includes the
study of the ``entropy'' of the solution, related to its logarithm,
the proof of the Harnack inequality for parabolic equations by Moser
\cite{moser} relies in an essential way on the observation that the
logarithm of the solution of a parabolic equation in divergence form
satisfies an equation with a dominating quadratic term in the left
hand side. This observation is then combined with a lemma that is the
parabolic counterpart of a result by F.~John and L.~Nirenberg about
functions with bounded mean oscillation. Soon afterwards,
S.~N.~Kruzhkov observes that the use of this lemma can be avoided
thanks to a Poincaré inequality due to Sobolev, see
\cite[Eq.~(1.18)]{kru63,kru64}.

\paragraph{Weak Harnack inequality.} Moser \cite{moser} and then
Trudinger \cite[Theorem 1.2]{zbMATH03256735} proved a weak Harnack
inequality for parabolic equations. Lieberman \cite{MR1465184} makes
the following comment: ``\textit{It should be noted [\dots] that Trudinger was
the first to recognize the significance of the weak Harnack inequality
even though it was an easy consequence of previously known
results.} [...] '' He also mentions that DiBenedetto and Trudinger
\cite{zbMATH03901241} showed that non-negative functions in the
elliptic De Giorgi class corresponding to super-solutions of elliptic
equations, satisfy a weak Harnack inequality and G.~L.~Wang
\cite{MR1032780,MR1246215} proves a weak Harnack inequality for
functions in the corresponding parabolic De Giorgi class.

\paragraph{Parabolic equations in non-divergence form.}
N.~V.~Krylov and M.~V.~Safonov \cite{MR563790} derived a Harnack
inequality for equations in non-divergence form. In order to do so,
they introduce a covering argument now known as the Ink-spots theorem,
see for instance \cite{zbMATH06233951}. Such a covering argument will
be later used in the various studies of elliptic equations in
divergence form, see \textit{e.g.} \cite{MR1032780,MR1246215} or
\cite{zbMATH03901241}.

\paragraph{Expansion of positivity.}
Ferretti and Safonov \cite{zbMATH01698711} establish the interior
Harnack inequality for both elliptic equations in divergence and
non-divergence form by establishing what they call growth lemmas,
allowing to control the behavior of solutions in terms of the measure
of their super-level sets.

Gianazza and Vespri introduce in \cite{MR2232212} suitable homogeneous
parabolic De Giorgi classes of order $p$ and prove a Harnack
inequality.  They shed light on the fact that their main technical
point is a \emph{expansion of positivity} in the following sense: if a
solution lies above $\ell$ in a ball $B_\rho(x)$ at time $t$, then it
lies above $\mu \ell$ in a ball $B_{2\rho} (x)$ at time $t+C \rho^p$
for some universal constants $\mu$ and $C$, that is to say constants
only depending on dimension and ellipticity constants. They 
mention that G.~L. Wang also used some expansion of positivity in
\cite{MR1032780}.

More recently, R.~Schwab and L.~Silvestre \cite{zbMATH06600511} used
such ideas in order to derive a weak Harnack inequality for parabolic
integro-differential equations with very irregular kernels.

\paragraph{Hypoellipticity.}

In the case where $A$ is the identity matrix, Equation~\eqref{e:main}
was first studied by Kolmogorov \cite{kolmogorov}. He exhibited a
regularizing effect despite the fact that diffusion only occurs in the
velocity variable. This was the starting point of the hypoellipticity
theory developed by H\"ormander \cite{MR0222474} for equations with
smooth variable coefficients (unlike $A$ and $B$ in \eqref{e:main}).

\paragraph{Regularity theory for ultraparabolic equations.}

The elliptic regularity for degenerate Kolmogorov equations in
divergence form with discontinuous coefficients, including
\eqref{e:main} with $B=0$, started at the end of the years 1990 with
contributions including
\cite{zbMATH00897958,zbMATH01220834,zbMATH01203628,zbMATH01610716}.
As far as the rough coefficients case is concerned, A.~Pascucci and
S.~Polidoro \cite{pp} proved that weak (sub)solutions of
\eqref{e:main} are locally bounded (from above). This result was later
extended in \cite{zbMATH05236721,zbMATH07140155}. Then W.~Wang and
L.~Zhang \cite{wz09,wz11,wz17} proved that solutions of \eqref{e:main} are
H\"older continuous. Even if the authors do not state their result as
an a priori estimate, it is possible to derive from their proof the
following result for a class of ultraparabolic equations. Such a class
containes equations of the form \eqref{e:main} with $B=0$.
More recently,  M.~Litsgård and K.~Nyström \cite{LN20} established existence and uniqueness results for the Cauchy Dirichlet problem for Kolmogorov-Fokker-Planck type equations with rough coefficients. 
\begin{theorem}[H\"older regularity -- \cite{wz09,wz11,wz17}] \label{t:holder}
  There exist $\alpha \in (0,1)$  only depending on dimension,
  $\lambda$ and $\Lambda$ such that all weak solution $f$  of
  \eqref{e:main} in some cylindrical open set $\Omega \supset Q_1=(-1,0] \times B_1 \times B_1$ satisfies
  \[
    [f]_{C^\alpha (Q_{1/2})} \le C (\|f\|_{L^2 (Q_1)} + \|S\|_{L^\infty(Q_1)})
  \]
  with $Q_{1/2} = (-1/4,0] \times B_{1/8} \times B_{1/2}$; the
  constant $C$ only depends on the dimension and the ellipticity constants
  $\lambda, \Lambda$.
\end{theorem}

\paragraph{Linear kinetic equations with rough coefficients.}

Since the resolution of the 19\textsuperscript{th} Hilbert problem by
E. de Giorgi \cite{DG57}, it is known that being able to deal with
coefficients that are merely bounded is of interest for studying
non-linear problems. There are several models from the kinetic theory
of gases related to equations of the form \eqref{e:main} with $A$, $B$
and $S$ depending on the solution itself. The most famous and
important example is probably the Landau equation
\cite{landau1980statistical}. 

An alternative proof of the H\"older continuity Theorem~\ref{t:holder}
was proposed by F.~Golse, C.~Mouhot, A.~F.~Vasseur and the second author
\cite{gimv} and a Harnack inequality was obtained.
\begin{theorem}[Harnack inequality -- \cite{gimv}] \label{t:harnack}
  Let $f$ be a non-negative weak solution of \eqref{e:main} in some
  cylindrical open set $\Omega \supset Q^0 := (-1,0] \times B_{R_0} \times
    B_{R_0}$. Then
  \[
    \sup_{Q_-} f \le C \left( \inf_{Q_+} f + \| S \|_{L^\infty(Q^0)} \right)
  \]
  where $Q_+ = (-\omega^2,0] \times B_{\omega^3} \times B_\omega$ and
  $Q_- = (-1,-1+\omega^2] \times B_{\omega^3} \times B_\omega$; the
  constants $C$ and $\omega$ only depend on the dimension and the ellipticity
  constants $\lambda, \Lambda$. 
\end{theorem}
Such a Harnack inequality implies in particular the strong maximum
principle \cite{zbMATH07118486} relying on a geometric construction
known as Harnack chains.  The H\"older regularity result of
\cite{wz09} was extended by Y. Zhu \cite{zhu2020velocity} to general
transport operators $\partial_t + b(v)\cdot \nabla_v$ for some non-linear
function $b$.

To finish with, we mention that C.~Mouhot and the second author
\cite{imbert2018toy} initiated the study of a toy non-linear model and
F. Anceschi and Y. Zhu continued it in \cite{anceschi-zhu}. Both
studies rely in an essential way on H\"older continuity of  weak
solutions to the linear equation \eqref{e:main}.

\paragraph{A functional analysis point of view.} 

The functional analysis framework used in \cite{gimv} was clarified by
S. Armstrong and J.-F. Mourrat in \cite{am19}. They show that it is
sufficient to control $f$ in $L^2_{t,x} H^1_v$ and
$(\partial_t+ v \cdot \nabla_x f) \in L^2_{t,x} H^{-1}_v$ (locally) to
derive new Poincaré inequalities.

There is an interesting bridge between the functional analysis point
of view and the PDE one. Indeed, under the assumptions the authors of
\cite{am19} work with, it is possible  to consider the Kolmogorov
equation
\[
  \LK f := (\partial_t + v \cdot \nabla_x f) - \Delta_v f = H
\]
with $H \in L^2_{t,x}H^{-1}_v$. In \cite{gimv}, the equation is
rewritten under the equivalent form $\LK f = \nabla_v \cdot H_1 +H_0$.
But in order to derive local properties of solutions such as their
H\"older continuity by elliptic regularity methods, it is necessary to
be able to work with \emph{sub-solutions} of the Kolmogorov
equation. In this case $\LK f = H - \mu$ where $\mu$ is an arbitrary
Radon measure. Such additional terms are difficult to deal with in the
functional analysis framework presented in \cite{am19}.  With the
partial differential point of view, comparison principles are used in
\cite{gimv} to locally gain some integrability for non-negative
sub-solutions.

\paragraph{Kinetic equations with integral diffusions.}

We would like to conclude this review of literature by mentioning the
weak Harnack inequality derived in \cite{is} for kinetic
equations. The proof also relies on De Giorgi type arguments that are
combined with a covering argument, referred to as an Ink-spots
theorem and inspired by the elliptic regularity for equations in
non-divergence form (see above). The interested reader is referred to the
introduction of \cite{is} for further details.

\subsection{Weak expansion of positivity}

The proof of the main result of this article relies on proving that
super-solutions of \eqref{e:main} expand positivity along times
(Lemma~\ref{l:pop}) and to combine it with the covering argument from
\cite{is} mentioned in the previous paragraph.  The derivation of the
weak Harnack inequality in the present article from the expansion of
positivity follows very closely the reasoning in \cite{is}.

In contrast with parabolic equations, it is not possible to apply the
Poincaré inequality in $v$ for $(t,x)$ fixed when studying solutions
of linear Fokker-Planck equations such as \eqref{e:main}. Instead, if
a sub-solution vanishes enough, then a quantity replacing the average in the usual Poincaré inequality is
decreased in the future. See $\theta_0 M$ in the weak Poincaré
inequality  in the next paragraph (Theorem~\ref{t:wpi}).

\begin{figure}
\centering{\includegraphics[height=3cm]{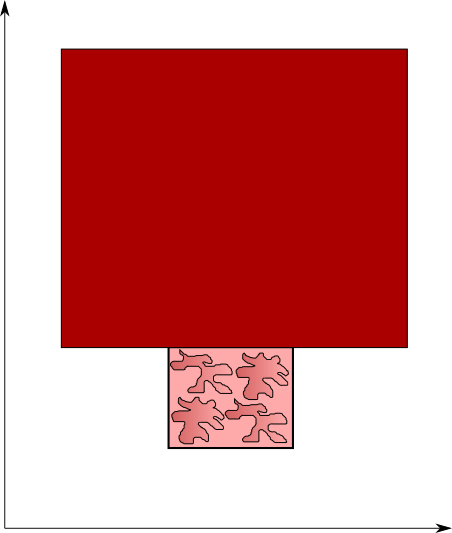}}
\put(-45,65){$Q_1$}
\put(-60,7){\small $\Qpos$}
\put(-10,-10){\scriptsize $(x,v)$}
\put(-80,80){\scriptsize $t$}
\caption{Expansion of positivity.}
\label{fig:intro}
\end{figure}
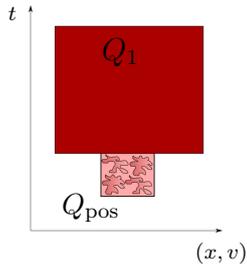
A way to circumvent this difficulty is to establish the expansion of
positivity of super-solutions in the spirit of \cite{zbMATH01698711}.
Given a small cylinder $\Qpos$ lying in the past of $Q_1$ (see
Figure~\ref{fig:intro}), Lemma~\ref{l:pop} states that if a
super-solution $f$ lies above $1$ in ``a good proportion'' of $\Qpos$,
then it lies above a constant $\ell_0>0$ in the whole cylinder
$Q_1$. Roughly speaking, such a lemma transforms an information in
measure about positivity in the past into a pointwise positivity in
the future in a (much) larger cylinder.

We emphasize the fact that in the classical parabolic case,
S.~N.~Kruzhkov does not need to prove such an expansion of positivity
thanks to an appropriate Poincar\'e inequality that can be applied
at any time $t>0$. Such an approach is inapplicable when there is
the additional  variable $x$.

Such a weak propagation of positivity was already proved in
\cite{gimv} thanks to a lemma of intermediate values, in the spirit of
De Giorgi's original proof. We recall that the proof of this key lemma
is not constructive in \cite{gimv}.  C.~Mouhot and the first author \cite{gm}
recently managed to make the proof of the intermediate value lemma
constructive.

\subsection{A weak Poincaré inequality}

The proof of the expansion of positivity relies on the following
weak Poincaré inequality.  The geometric setting is shown in
Figure~\ref{fig:wpi}. 
\begin{theorem}[Weak Poincar\'e inequality]
  \label{t:wpi}
  Let $\eta \in (0,1)$.  There exist $R > 1$ and $\theta_0 \in (0,1)$ depending on dimension and $\eta$ such that, if  $\Qext = (-1-\eta^2,0] \times B_{8R} \times B_{2R}$ and  $\Qzero = (-1-\eta^2,-1] \times B_{\eta^3} \times B_{\eta} $ (see  Figure~\ref{fig:wpi}), then for any non-negative function  $f \in L^2 (\Qext)$ such that $\nabla_v f \in L^2 (\Qext)$, $(\partial_t  +v\cdot \nabla_x) f \in L^2((-1-\eta^2,0] \times B_{8R},H^{-1}(B_{2R}))$, $f \le M $ in  $Q_1$ and  $|\{ f = 0 \} \cap \Qzero|  \ge \frac14 |\Qzero|, $ satisfying
\begin{align*}
(\partial_t + v \cdot \nabla_x) f \le H &\quad  \mbox{ in } \mathcal{D}' (\Qext) \quad \mbox{ with } H \in L^2_{t,x} H^{-1}_v(\Qext),
\end{align*} 
we have
  \[
    \| (f -\theta_0 M)_+ \|_{L^2 (Q_1)} \le  C (\|\nabla_v f \|_{L^2 (\Qext)} + \|H\|_{L^2_{t,x} H^{-1}_v(\Qext)})
  \]
  for some constant $C>0$ only depending on dimension.
  \end{theorem}
  \begin{remark}
    In the previous statement, $L^2_{t,x} H^{-1}_v(\Qext)$ is a short
    hand notation for    $L^2((-1-\eta^2,0] \times B_{8R}, H^{-1}(B_{2R}))$.
  \end{remark}
  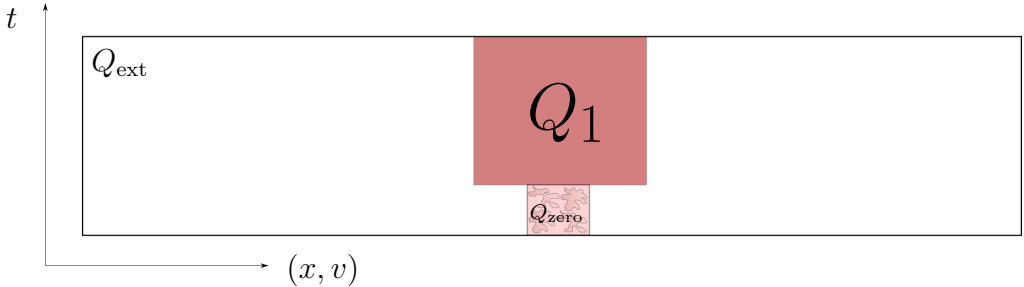
\begin{figure}[h]
\centering{\includegraphics[height=3.5cm]{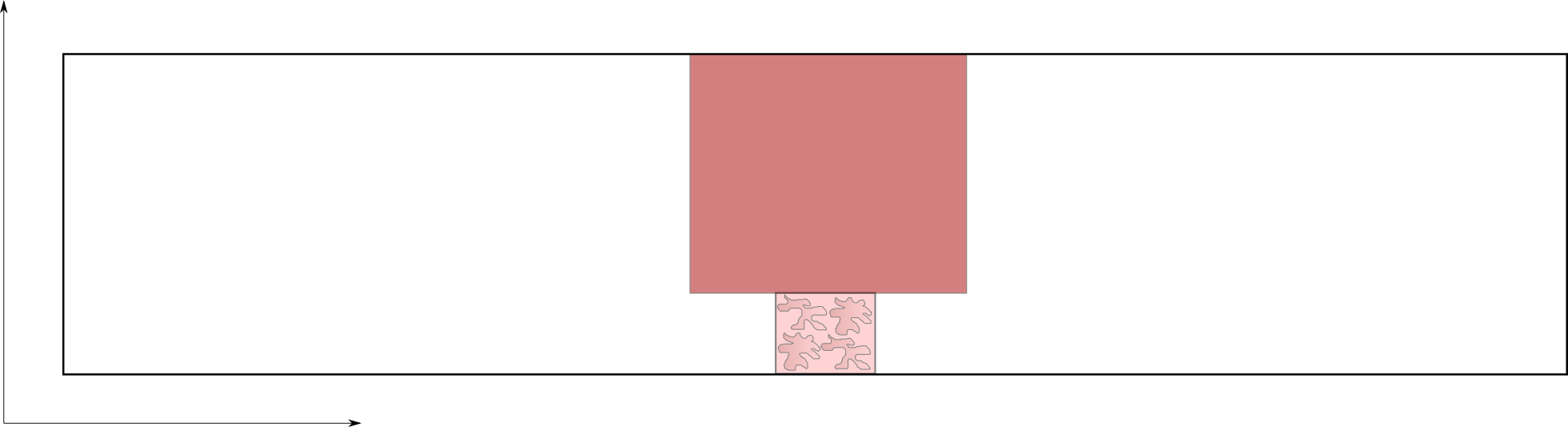}}
\put(-380,90){$t$}
\put(-348,74){$\Qext$}
\put(-185,52){\huge $Q_1$}
\put(-184,18){{\scriptsize $\Qzero$}}
\put(-275,-4){$(x,v)$}
\caption{Geometric setting of the weak Poincaré inequality.  If the
  set where the function $f$ vanishes in $\Qzero$ has a measure at
  least equal to $\frac14 |\Qzero|$, then
  $h \le \theta_0 \sup_{Q_1} f$ in the red cylinder $Q_1$. }
  \label{fig:wpi}
\end{figure}
A somewhat similar inequality was introduced by W.~Wang and L.~Zhang
for sub-solutions of ultraparabolic equations, see for instance
\cite[Lemmas~3.3 \& 3.4]{wz09} and the corresponding lemmas in
\cite{wz11,wz17}. Even if statements and proofs look different, they
share many similarities.  The main difference between statements
  comes from the fact that we adopt the functional framework from
  \cite{am19} and forget about the equation under study.  The main
difference in proofs lies on the fact that we avoid using repeatedly
the exact form of the fundamental solution of the Kolmogorov equation
and we seek for arguments closer to the classical theory of parabolic
equations presented for instance in \cite{L1} or \cite{MR1465184}.  In
contrast with \cite{wz09,wz11,wz17}, the information obtained through
the log-transform is summarized in only one weak Poincar\'e inequality
(while it is split it several lemmas in \cite{wz09,wz11,wz17}) and the
geometric settings of the main lemmas are as simple as possible. For
instance, it is the same for the weak Harnack inequality and for the
lemma of expansion of positivity (Lemma~\ref{l:pop}). We also mostly
use cylinders respecting the invariances of the equation (see the
defintion of $Q_r(z)$ in the paragraph devoted to notation), except
the ``large'' cylinders where the equation is satisfied such as
$\Qext$ in Theorem~\ref{t:wpi}.

\paragraph{The Lie group structure.}

Eq.~\eqref{e:main} is not translation invariant in the velocity
variable because of the free transport term. But this (class of)
equation(s) comes from mathematical physics and it enjoys the Galilean
invariance: in a frame moving with constant speed $v_0$, the equation
is the same. For $z_1 =(t_1,x_1,v_1)$ and $z_2 = (t_2,x_2,v_2)$, we
define the following non-commutative group product
\[ z_1 \circ z_2 = (t_1+t_2,x_1+x_2 + t_2 v_1, v_1 + v_2). \]
In particular, for $z = (t,x,v)$, the inverse element is
$z^{-1} = (-t,-x+tv,-v)$.

\paragraph{Scaling and cylinders.} Given a parameter $r>0$, the class
of equations~\eqref{e:main} is invariant under the scaling
\[
  f_r (t,x,v) = f(r^2 t, r^3x, r v ) .
\]
It is convenient to write $S_r(z) = (r^2t,r^3x,rv)$ if $z = (t,x,v)$.
It is thus natural to consider the following cylinders ``centered'' at
$(0,0,0)$ of radius $r>0$:
$Q_r = (-r^2,0] \times B_{r^3} \times B_r = S_r (Q_1)$. Moreover, in
view of the Galilean invariance, it is then natural to consider
cylinders centered at $z_0 \in \R^{1+2d}$ of radius $r>0$ of the form:
$Q_r (z_0) := z_0 \circ Q_r $ which is 

\begin{align*}
  Q_r(z_0)
  &:= \left\{ z \in \R^{1+2d} \ : \ z_0^{-1} \circ z  \in
    Q_{r}(0) \right\} \\
  &:= \left\{ -r^2 <t-t_0 \le 0, \ |x - x_0- (t-t_0)v_0|
    < r^3, \ |v-v_0| < r \right\}.
\end{align*}

\paragraph{Organization of the article.} In Section~\ref{s:weak}, the
definition of weak sub-solutions, super-solutions and solutions for
\eqref{e:main} is recalled and two properties of the log-transform are
given. Section~\ref{s:wpi} is devoted to the proof of the weak
Poincaré inequality. In Section~\ref{s:dg}, we explain how to derive
the lemma of expansion of positivity from the weak Poincaré inequality
and how to prove the weak Harnack inequality from expansion of
positivity by using a covering lemma called the Ink-spots
theorem. This last result is recalled in Appendix~\ref{a:ink}. In
another appendix, see \S\ref{a:holder}, we recall how H\"older
regularity can be derived directly from the expansion of positivity of
super-solutions. The proof of a technical lemma about stacked cylinders
is given in Appendix~\ref{a:cylinder}. 

\paragraph{Notation.} The open ball of the Euclidian space centered at
$c$ of radius $R$ is denoted by $B_R (c)$. The measure of a Lebesgue
set $A$ of the Euclidian space is denoted by $|A|$. The $z$ variable
refers to $(t,x,v) \in \R \times \R^d \times \R^d =\R^{1+2d}$.  For
$z_1,z_2 \in \R^{1+2d}$, $z_1 \circ z_2$ denotes their Lie group
product and $z_1^{-1}$ denotes the inverse of $z_1$ with respect to
$\circ$. For $r>0$, $S_r$ denotes the scaling operator. A constant is
said to be \emph{universal} if it only depends on dimension and the
ellipticity constants $\lambda,\Lambda$ appearing in
\eqref{e:ellipticity}. The notation $a \lesssim b$ means that
$a \le Cb$ for some universal constant $C>0$.

For an open set $\Omega$, $\D (\Omega)$ denotes the set of $C^\infty$
functions compactly supported in $\Omega$ while $\D' (\Omega)$ denotes
the set of distributions in $\Omega$.

\paragraph{Acknowledgements.} The authors are indebted to
C.~Mouhot for the fruitful discussions they had together during the
preparation of this paper.  The first author acknowledges
funding by the ERC grant MAFRAN 2017-2022.

\section{Weak solutions and Log-transform}
\label{s:weak}

\subsection{Weak solutions}

We start with the definition of weak (super- and sub-) solutions of \eqref{e:main}.
\begin{definition}[Weak solutions] \label{d:weak} Let $\Omega = I \times B^x \times B^v$ be open. A function
  $f\colon \Omega \to \R$ is a \emph{weak super-solution} (resp. \emph{weak sub-solution}) of \eqref{e:main}
  in  $\Omega$ if $f \in L^{\infty} (I, L^2 (B^x \times B^v)) \cap L^2 (I \times B^x, H^1(B^v))$ and
  $(\partial_t + v \cdot \nabla_x) f \in L^2(I \times B^x, H^{-1}(B^v))$ and
  \[ 
    - \int  f (\partial_t + v \cdot \nabla_x) \varphi \dz
    + \int  A \nabla_v f \cdot  \nabla_v \varphi  \dz - \int (B \cdot \nabla_v f+S) \varphi  \dz \ge 0
    \quad \text{(resp. $\le 0$)}
  \]
  for all non-negative $\varphi \in \D(\Omega)$. It is a \emph{weak solution} of \eqref{e:main} in
  $Q$ if it is both a weak super-solution and a  weak sub-solution.
\end{definition}

\bigskip

As explained in the introduction, the local boundedness of
sub-solutions has been known since \cite{pp}. We give below the
version contained in \cite{gimv}.
\begin{proposition}[Local upper bound for sub-solutions --
  \cite{gimv}] \label{p:l2linfty} Consider two cylinders
  $\Qint = (t_1,T] \times B_{r_x} \times B_{r_v}$ and
  $\Qext = (t_0,T] \times B_{R_x} \times B_{R_v}$ with $t_1 > t_0$,
  $r_x <R_x$ and $r_v < R_v$.  Assume that $f$ is a sub-solution of
  \eqref{e:main} in some cylindrical open set $\Omega \supset \Qext$. Then
\[ \esssup_{\Qint} f  \le C ( \| f_+ \|_{L^2(\Qext)} + \|S\|_{L^\infty (\Qext)}) \]
where $C$ only depends on $d,\lambda,\Lambda$ and $(t_1 -t_0,R_x-r_x,R_v-r_v)$. 
\end{proposition}

\subsection{Log-transform of sub-solutions}

For technicals reasons, the positive part of the opposite of the
logarithm is replaced with a more regular function $G$ that keeps the
important features of $\max(0,-\ln)$.  The function $\max(0,-\ln)$ was
first considered in \cite{kru63,kru64}.
\begin{lemma}[A convex function] \label{l:G} There exists
  $G: (0,+\infty) \to [0,+\infty)$ non-increasing and $C^2$ such that
\begin{itemize}
\item $G'' \ge (G')^2 \text{ and } G' \le 0 \text{ in } (0,+\infty)$,
\item $G$ is supported in $(0,1]$,
\item $G (t) \sim - \ln t$ as $t \to 0^+$, 
\item $-G'(t) \le \frac1t$  for $t \in (0,\frac14]$.
\end{itemize}
\end{lemma}
\begin{lemma}[Log-transform of solutions]\label{l:logtrans}
  Let $\eps \in (0,\frac14]$ and $f$ be a non-negative weak
  super-solution of \eqref{e:main} in a cylinder
  $\Qext = (t_0,T] \times B_{R_x} \times B_{R_v}$.  Then
  $g = G(\eps+f)$ satisfies
\[ 
  (\partial_t + v \cdot \nabla_x ) g + \lambda |\nabla_v g|^2 \le
  \nabla_v \cdot (A \nabla_v g) + B \cdot \nabla_v g + \eps^{-1} |S|  \quad \text{ in } \Qext,
\]
\textit{i.e.} it is a sub-solution of the corresponding equation in $\Qext$. 
\end{lemma}
\begin{proof}
  We first note that $g \in L^\infty (\Qext)$ since
  $0 \le g \le G(\eps)$.  Moreover,
  $\nabla_v g = G'(\eps + f) \nabla_v f$ with
  $|G'(\eps+f)| \le |G'(\eps)|$. In particular,
  $g \in L^2 ((t_0,T] \times B_{R_x}, H^1 (B_{R_v}))$.  In order to
  obtain the sub-equation, it is sufficient to consider the
  test-function $G'(\eps+f) \Psi$ in the definition of super-solution
  for $f$.
\end{proof}
The following observation is key in Moser's reasoning since the \emph{square}
of the $L^2$-norm of $\nabla_v g$ is controlled by the mass of $g$.
\begin{lemma}[Time evolution of the mass of $g$]\label{l:ee-g}
  Let $f$ be a non-negative weak sub-solution of \eqref{e:main} in a
  cylinder $\Qext = (t_0,T] \times B_{r_x} \times B_{r_v}$ and
  $\Qint = (t_1,T] \times B_{R_x} \times B_{R_v}$ with $t_1 > t_0$,
  $r_x < R_x$ and $r_v < R_v$. Then
\[
\frac{\lambda}2 \int_{\Qint} |\nabla_v g |^2 
\le C \left( \int_{\Qext} g(\tau)  + 1  + \eps^{-1} \|S\|_{L^\infty(\Qext)} \right)
\]
where $C$  depends on dimension, $\lambda$, $\Lambda$, $\Qint$ and $\Qext$.
\end{lemma}
\begin{proof}
  Consider a smooth cut-off function $\Psi$ valued in $[0,1]$,
  supported in $\Qext$ and equal to $1$ in $\Qint$ and use $\Psi^2$ as
  a test-function for the sub-equation satisfied by $g$ and get
  \begin{align*}
   \lambda \int |\nabla_v g|^2 \Psi^2   
     \le &- 2 \int A \nabla_v g \cdot \Psi \nabla_v \Psi + \int g (\partial_t + v \cdot \nabla_x) \Psi^2  + \int (B \cdot \nabla_v g + \eps^{-1} |S| ) \Psi^2 \\
    \le & \frac\lambda4 \int |\nabla_v g |^2 \Psi^2 + C (1+\int_{\Qext} g) + \frac\lambda4 \int |\nabla_v g |^2 \Psi^2 + C \eps^{-1} \|S\|_{L^\infty(\Qext)}.
  \end{align*}
  This yields the desired estimate. 
\end{proof}

\section{A weak Poincar\'e inequality}
\label{s:wpi}

In order to prove Theorem~\ref{t:wpi}, we first derive a local
Poincaré inequality (Lemma~\ref{l:poincare-local}) with an error
function $h$ due to the localization. This function $h$ satisfies
\begin{equation}
  \label{e:cauchy}
  \begin{cases}
    \LK h = f \LK \Psi, & \text{ in } (a,b) \times \R^{2d},\\
    h = 0, & \text{ in } \{a \} \times \R^{2d},
  \end{cases}
\end{equation}
where $\LK = (\partial_t + v\cdot \nabla_x) - \Delta_v$ is the
Kolmogorov operator and $\Psi$ is a cut-off function equal to $1$ in
$Q_1$.
We will estimate $h$ in Lemma~\ref{l:max} below.
\begin{lemma}[A local estimate]
  \label{l:poincare-local}
  Let $\Qext = (a,0] \times B_{R_x} \times B_{R_v}$ be a cylinder such
  that $\overline{Q_1} \subset \Qext$, and let $\Psi \colon \R^{2d+1} \to [0,1]$
  be $C^\infty$, supported in $\Qext$ and $\Psi=1$ in $Q_1$.  Then for any
  function $f \in L^2 (\Qext)$ such that $\nabla_v f \in L^2 (\Qext)$ and
  $(\partial_t + v \cdot \nabla_x) f \le H$ in $\mathcal{D}' (\Qext)$
  with $H \in L^2 ((a,0] \times B_{R_x} ,H^{-1}(B_{R_v}))$, we have
  \[
    \| (f -h)_+ \|_{L^2 (Q_1)} \le C (\|\nabla_v f \|_{L^2 (\Qext)} + \|H \|_{L^2 ((a,0] \times B_{R_x} ,H^{-1} (B_{R_v}))}
  \]
  where $h$ satisfies the  Cauchy problem \eqref{e:cauchy} and  $C = c(a) (1+\|\nabla_v \Psi\|_\infty)$ for some constant $c(a)$ only  depending on $|a|$.
\end{lemma}
\begin{proof}
  Since $H \in L^2((a, 0]\times B_{R_x},H^{-1} (B_{R_v}))$, there  exists $H_0,H_1 \in L^2 (\Qext)$ such that  $H = \nabla_v \cdot H_1 + H_0$ and such that  $ \|H_0 \|_{L^2(\Qext)}+\|H_1 \|_{L^2(\Qext)}\leq 2\|H \|_{L^2    ((a,0] \times B_{R_x} ,H^{-1} (B_{R_v}))}$ (see for instance \cite{zbMATH01181255}).  The function $g = f \Psi$ satisfies
\[
  \LK g \le  \nabla_v \cdot \tilde H_1 + \tilde H_0 + f [(\partial_t + v \cdot \nabla_x) \Psi -  \Delta_v \Psi]
  \text{ in }   \mathcal{D}'((a,0) \times \R^{2d})
\]
with $\tilde H_1 = (H_1- \nabla_v f) \Psi$ and $\tilde H_0 = H_0 \Psi - H_1 \nabla_v \Psi - \nabla_v \Psi \cdot  \nabla_v f$. We thus get
\[ \LK (g-h)  \le \tilde H  \text{ in } \mathcal{D}'((a,0) \times \R^{2d}) 
\]
with $\tilde H = \nabla_v \cdot \tilde H_1 + \tilde H_0$. 
We then multiply by $(g-h)_+$ to get the natural energy estimate for all $T, T'\in (a,0)$ and $\eps >0$,
\begin{align*}
     \int (g-h)_+^2 (T,x,v) \dx \dv & + \int_{a}^{T'} \int |\nabla_v (g-h)_+ |^2 \dt \dx \dv \\
    \le &  2 \int_{a}^{T} \int |-\tilde H_1 \cdot \nabla_v (g-h)_+ + \tilde H_0 (g-h)_+| \dt \dx \dv \\
    \le &   \frac12 \| \nabla_v (g-h)_+\|^2_{L^2 ((a,0] \times \R^{2d})}+ 2\|\tilde H_1\|^2_{L^2 ((a,0] \times \R^{2d})}\\ 
    &+ 2\eps \| (g-h)_+ \|^2_{L^2 ((a,0] \times \R^{2d})} + \frac1{2\eps} \| \tilde H_0 \|^2_{L^2 ((a,0] \times \R^{2d})}.
  \end{align*}
Remark that we can deal with the two terms of the left hand side separately so that we can consider the two parameters $T$ and $T'$.
Writing $\|\cdot\|_{L^2}$ for $\|\cdot \|_{L^2 ((a,0) \times \R^{2d})}$, we get
  after integrating in $T$ from $a$ to $0$,  and choosing $T'=0$
  \[ \| (g-h)_+\|_{L^2}^2 \le -2\eps a
  \| (g-h)_+\|^2_{L^2 } - a \|\tilde H_1\|^2_{L^2} -
  \frac{a}{2\eps} \|\tilde H_0\|^2_{L^2 } .\]
Then remark that the function $(g-h)_+$ equals $(f -h)_+$ in $Q_1$ and that
\begin{align*}
  \| \tilde H_1 \|_{L^2((a,0] \times \R^{2d})} & \le \|H_1\|_{L^2(\Qext)} + \|\nabla_v f \|_{L^2(\Qext)}, \\
  \| \tilde H_0 \|_{L^2((a,0] \times \R^{2d})} & \le \|H_0\|_{L^2(\Qext)} + \|\nabla_v \Psi\|_\infty
                                       ( \|H_1\|_{L^2(\Qext)} + \|\nabla_v f \|_{L^2(\Qext)} ).
\end{align*}
We get the desired inequality by combining the three previous inequalities and choosing $\eps = -(4a)^{-1}$.
\end{proof}

In view of Lemma~\ref{l:poincare-local}, it is sufficient to prove
that if the function $f$ satisfies
\[|\{ f = 0 \} \cap \Qzero|  \ge \frac14 |\Qzero|, \]
then the
 function $h$ given by the Cauchy problem~\eqref{e:cauchy} is bounded from above by
$\theta_0 M$ for some universal parameter $\theta_0 \in (0,1)$.
\begin{lemma}[Control of the localization term]\label{l:max}
  Let $\eta \in (0,1]$.  There exist a (large) constant $R>1$ and a (small) constant  $\theta_0 \in (0,1)$ both depending on the dimension and $\eta$,  and a $C^\infty$ cut-off function  $\Psi \colon \R^{2d+1} \to [0,1]$, supported in $\Qext = (-1-\eta^2,0] \times B_{8R} \times B_{2R}$ and equal to $1$ in $Q_1$, such that for all non-negative bounded function
  $f: \Qext \to \R$ satisfying
  \begin{equation}
    \label{e:z-set}
    |\{ f = 0 \} \cap \Qzero|    \ge \frac14 |\Qzero|,
  \end{equation}
  the solution $h$ of the following initial value problem
  \[
      \begin{cases} 
        \LK h = f \LK \Psi & \text{ in } (-1-\eta^2,0) \times \R^{2d} \\
        h = 0 & \text{ in } \{ -1-\eta^2 \} \times \R^{2d} 
      \end{cases}
  \]
  satisfies: $h \le \theta_0 \|f\|_{L^\infty(\Qext)}$ in  $Q_1$.
\end{lemma}
\begin{remark}
  This lemma is related to \cite[Lemma~3.4]{wz09} and
  \cite[Lemma~3.3]{wz11}.
\end{remark}
\begin{remark}
  The conclusion of the lemma and its proof are essentially unchanged
  under the weaker assumption
  \( |\{ f = 0 \} \cap \Qzero| \ge \alpha_0 |\Qzero| \) for some
  $\alpha_0 \in (0,1)$. 
\end{remark}
\begin{remark}
Theorem \ref{t:wpi} will be used in the proof of
  Lemma~\ref{l:pop} about the expansion of positivity of
  super-solutions. The parameter $\eta$ will be then chosen after choosing $\theta$.  
\end{remark}
The proof of Lemma~\ref{l:max} requires the following test-function
whose construction is elementary.
\begin{lemma}[Cut-off function]\label{l:cutoff}
  Given $\eta \in (0,1]$ and $T \in (0,\eta^2)$, there exists a smooth
  function
  $\Psi_1 : [-1-\eta^2,0] \times \R^{d}\times \R^d \to [0,1]$,
  supported in $[-1-\eta^2,0] \times B_{8} \times B_2$, equal to
  $1$ in $(-1,0] \times B_{1} \times B_{1}$, such that
  $(\partial_t + v \cdot \nabla_x) \Psi_1 \ge 0$ everywhere and
  $(\partial_t + v \cdot \nabla_x) \Psi_1 \ge 1$ in
  $(-1-\eta^2,-1-T] \times B_{1} \times B_{1}$.
\end{lemma}
\begin{proof}
  Consider
  $\Psi_1 (t,x,v) = \varphi_1 (t) \varphi_2 (x-tv) \varphi_3 (v)$ with
\begin{itemize}
\item a smooth function $\varphi_1: [-1-\eta^2,0] \to [0,1]$ equal to $1$ in $[-1,0]$ with $\varphi_1 (-1-\eta^2) = 0$,
  $\varphi_1' \ge 0$ in $[-1-\eta^2,0]$ and $\varphi_1'=1$ in  $[-1-\eta^2,-1-T]$;
\item a smooth function  $\varphi_2 : \R^d \to [0,1]$ supported in $B_{4}$ and equal to $1$ in $B_3$;
\item a smooth function  $\varphi_3 : \R^d \to [0,1]$ supported in $B_{2}$ and equal to $1$ in $B_1$.
\end{itemize}
It is then easy to check that the conclusion of the lemma holds true. 
\end{proof}
We can now turn to the proof of Lemma~\ref{l:max}.
\begin{proof}[Proof of Lemma~\ref{l:max}]
  If $f=0$ in $\Qext$, then $h=0$. We thus can assume from
  now on that $f$ is not identically $0$. We next reduce to the case
  $\|f\|_{L^\infty(\Qext)}=1$ by considering $f/\|f\|_{L^\infty(\Qext)}$.

  We  introduce a time lap $T$ between the top of the cylinder $\Qzero$ and the bottom of the
  cylinder $Q_1$, see Figure~\ref{fig:1}. 
  \begin{figure}
\centering{\includegraphics[height=2.5cm]{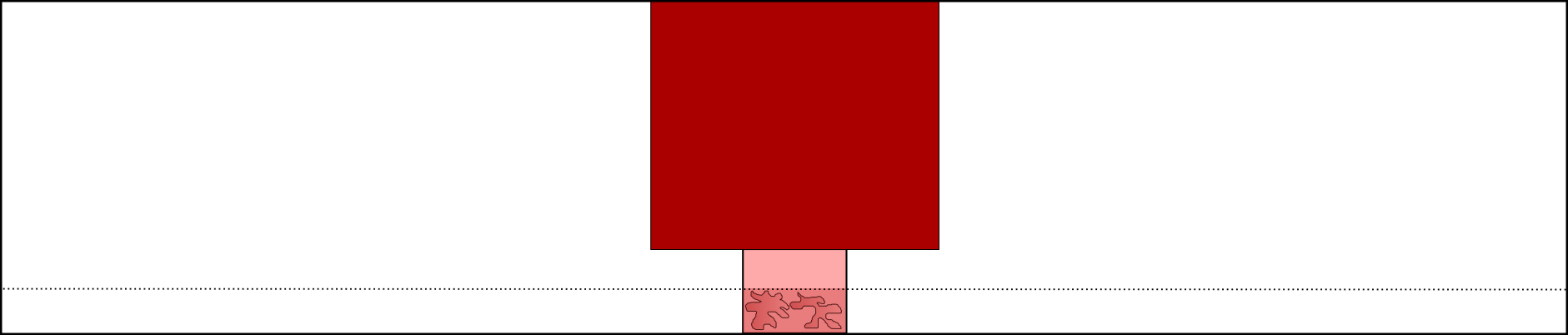}}
\put(-330,59){$\Qext$}
\put(-175,40){\huge $Q_1$}
\put(-62,14){$t=-1-T$}
\put(-203,5){$\Qzero$}
\caption{Reducing to the case with a time lap.}
\label{fig:1}
\end{figure}
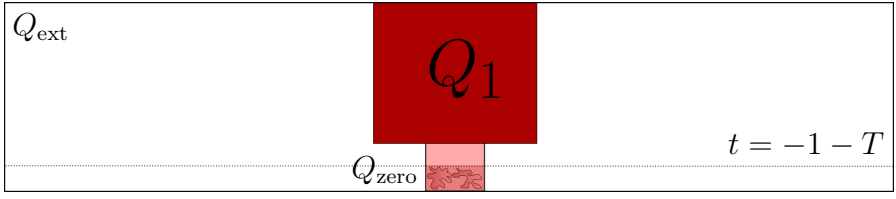

Fix $T=\eta^2/8$. Since $|\Qzero \cap \{ t \ge -1-T\}| = \frac18|\Qzero|$, then
  \begin{equation}
    \label{e:z-set-bis}
    |\{ f = 0 \} \cap \Qzero \cap \{t \le -1-T\}|    \ge \frac18 |\Qzero|.
  \end{equation}
  
Let $R>1$ to be chosen later. We consider the cut-off function
\[
 \Psi (t,x,v) = \Psi_1 (t,x/R,v/R).
\]
  Remark that $\Psi$ is supported in $\Qext$ and equal to $1$ in  $(-1,0] \times B_{R} \times B_{R}$. Moreover,
  \[
    \LK \Psi (t,x,v) = (\partial_t + v \cdot \nabla_x) \Psi_1 (t,x/R,v/R) - R^{-2} \Delta_v \Psi_1 (t,x/R,v/R),
  \]
  where in the last equation $v \cdot \nabla_x$ means the scalar product of the value of third variable in $\Psi$ (here it is $\frac{v}{R}$) with the gradient in the second variable.  
  We then have
  \[
    \LK (h- \Psi) = -(1-f)  (\partial_t + v \cdot \nabla_x) \Psi_1 (t,x/R,v/R) + \frac{1-f}{R^2} \Delta_v \Psi_1 (t,x/R,v/R)
  \]
  and we can write
  \[h - \Psi = -P_R + E_R \] ($P_R$ for positive and $E_R$ for error)
  with $P_R$ and $E_R$ solutions of the following Cauchy problems in
  $(-1-\eta^2,0) \times \R^{2d}$,
  \begin{align*}
    \LK P_R &= (1-f) (\partial_t + v \cdot \nabla_x) \Psi_1(t,x/R,v/R), \\
    \LK E_R &= \frac{1-f}{R^2} \Delta_v \Psi_1 (t,x/R,v/R),
  \end{align*}
  and $P_R = E_R =0$ at time  $t=-1-\eta^2$.
  \bigskip
   
  We claim that there exist constants $C>0$ and  $\delta_0 >0$ depending on the dimension and $\eta$ (in particular independent of $R$) such that
\begin{equation}\label{e:claim} 
E_R \le C R^{-2} \quad \text{ and } \quad P_R \ge \delta_0 \quad  \text{ in } Q_1.
\end{equation}
As far as the estimate of $E_R$ is concerned, it is enough to remark
that $\LK E_R \le C_0 R^{-2}$ for some constant
$C_0= \| \Delta_v \Psi_1 \|_{L^\infty}$ only depending on
$d,\lambda,\Lambda, \eta$ (in particular not depending on $R$). The
maximum principle then yields the result for some universal constant
$C$.  As far as $P_R$ is concerned, we remark that
\[
\LK P_R \ge   \un_{\mathcal{Z}} \quad \text{ in } (-1-\eta^2,0] \times \R^{2d}
\]
where $\mathcal{Z} = \{f = 0\} \cap \Qzero \cap \{ t \le -1-T\}$.  We
use here the fact that $(\partial_t + v \cdot \nabla_x) \Psi_1 \ge 1$
in $(-1-\eta^2,-1-T) \times \R^{2d}$.  Let $P$ be such that
$\LK P = \un_{\mathcal{Z}}$ in $(-1-\eta^2,0] \times \R^{2d}$ and
$P=0$ at the initial time $-1-\eta^2$. The maximum principle implies
that $P_R \ge P$ in the time interval $(-1-\eta^2,0]$, and in
particular in $Q_1$. The strong maximum principle implies that
$P \ge \delta_0$ in $Q_1$ for some constant
$\delta_0>0$ depending on the dimension and $\eta$. Indeed, one can use the fundamental solution $\Gamma$ of
the Kolmogorov equation and write
\[ P (t,x,v) = \int \Gamma (z,\zeta) \un_{\mathcal{Z}} (\zeta) \dd
  \zeta \ge \frac18 \mathsf{m} |\Qzero| = \delta_0,\]
  with $\mathsf{m}= \min\limits_{Q_1 \times \Qzero\cap \{ t \le -1-T\}} \Gamma.$ The claim \eqref{e:claim} is now proved.  \bigskip

Inequalities from \eqref{e:claim} imply that
  \[ h \le 1 - \delta_0 + C R^{-2} \text{ in } Q_1.\] This yields the
  desired result with $\theta_0 = 1-\delta_0/2$ for $R$ large
  enough. Note in particular that $R$ only depends on $C$ and
  $\delta_0$ and consequently depends on the dimension and $\eta$.
\end{proof}

\section{The weak Harnack inequality}
\label{s:dg}

Before proving the weak Harnack inequality stated in
Theorem~\ref{t:weak-harnack}, we investigate how Eq.~\eqref{e:main}
expands  positivity of super-solutions.
\begin{figure}[h]
\centering{\includegraphics[height=3.5cm]{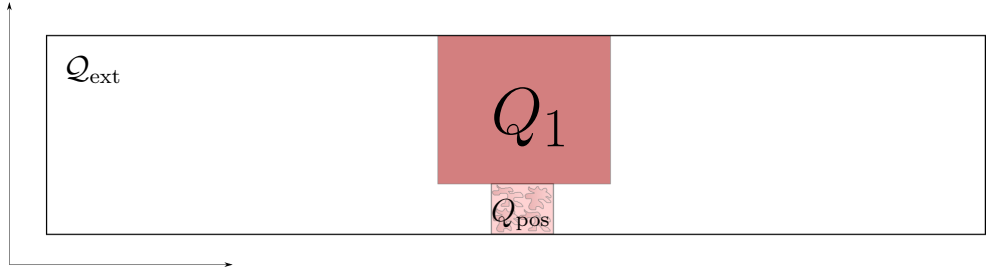}}
\put(-345,71){$\Qbext$}
\put(-185,50){\huge $Q_1$}
\put(-185,17){$\Qpos$}
\caption{Geometric setting of the expansion of positivity lemma. It is
  the same as the one of the weak Poincaré inequality, except that
  $\Qext$ and $\Qzero$ are replaced with $\Qbext$ and $\Qpos$. Here,
  $\Qpos$ denotes a set ``in the past'' where $\{f \ge 1\}$ occupies
  half of it.}
  \label{fig:pop}
\end{figure}
\begin{lemma}[Expansion of positivity]\label{l:pop}
  Let $\theta \in (0,1]$ and
  $\Qpos = (-1-\theta^2,-1] \times B_{\theta^3} \times B_{\theta}$ and
  let $R$ be the constant given by Lemma~\ref{l:max} which depends on
  $\theta$, $d$, $\lambda$, $\Lambda$. There exist
  $\eta_0,\ell_0 \in (0,1)$, only depending on $\theta$, $d$,
  $\lambda$, $\Lambda$, such that for
  $\Qbext := (-1-\theta^2,0] \times B_{9 R} \times B_{3 R}$ and any
  non-negative super-solution $f$ of \eqref{e:main} in some
  cylindrical open set $\Omega \supset \Qbext$ and
  $\|S\|_{L^\infty (\Qbext)} \le \eta_0$ and such that
  \( |\{ f \ge 1 \} \cap \Qpos | \ge \frac12 |\Qpos|,\) we have
  $f \ge \ell_0$ in $Q_1$.
\end{lemma}
\begin{remark}
  The parameter $\theta$ will be chosen in such a way that the stacked
  cylinder $\overline \Qpos^m$ is contained in $Q_1$(see the
  definition of stacked cylinders in Appendix \ref{a:ink}).  The
  cylinder ${\overline \Qpos}^m$ can be thought of as the union of $m$
  copies of $\Qpos$ stacked above (in time) of $\Qpos$.  Such
  cylinders are used in the covering argument used in the proof of the
  weak Harnack inequality and the parameter $m \in \mathbb{N}$ only
  depends on dimension.
\end{remark}
\begin{proof}[Proof of Lemma~\ref{l:pop}]
  We consider $g = G(f+\eps)$. We remark that $g \le G(\eps)$ since
  $f$ is non-negative and $G$ is non-increasing. We also remark that
$|G'(f+\eps)|\le |G'(\eps)| \le \eps^{-1}$ since $f \ge 0$, see Lemma~\ref{l:G}.

  We know from Lemma~\ref{l:logtrans} that $g$ is a non-negative
  sub-solution of \eqref{e:main} with $S$ replaced with
  $S G'(f+\eps)$. In particular,
  $(\partial_t + v \cdot \nabla_x )g \le H$ with
  $H = \nabla_v \cdot (A \nabla_vg) + B\cdot \nabla_v g+ \eps^{-1}|S|$.
  Recall that for a set $Q \subset \R^{2d+1}$, $S_{(r)}(Q)=\{(r^2t,r^3x,rv) \quad \mbox{for } z=(t,x,v)\in Q \}. $
We introduce $\eta \in (0,\frac{\theta}{2})$ and $\iota >0$ two parameters depending on $\theta$ to chosen later in the proof. 

We are going to apply successively: the $L^2-L^\infty$ estimate from  $Q_1$ to a slightly larger cylinder $Q_{1+\iota}$ for an accurate choice of $\iota$; the (scaled) weak  Poincaré inequality in the big cylinder  $\Qtext = S_{(1+\iota)}(\Qext)$ with $\Qext = (-1-\eta^2, 0] \times B_{8R} \times B_{2R}$.  Then  we estimate the $L^2$-norm of $\nabla_v g$ by the square root of its  mass in a cylinder larger than $\Qtext$, namely  $S_{(1+\iota)^2}(\Qext)\subset \Qbext$. This is illustrated in Figure~\ref{fig:pop-plus}.
\begin{figure}[h]
\centering{\includegraphics[height=2.5cm]{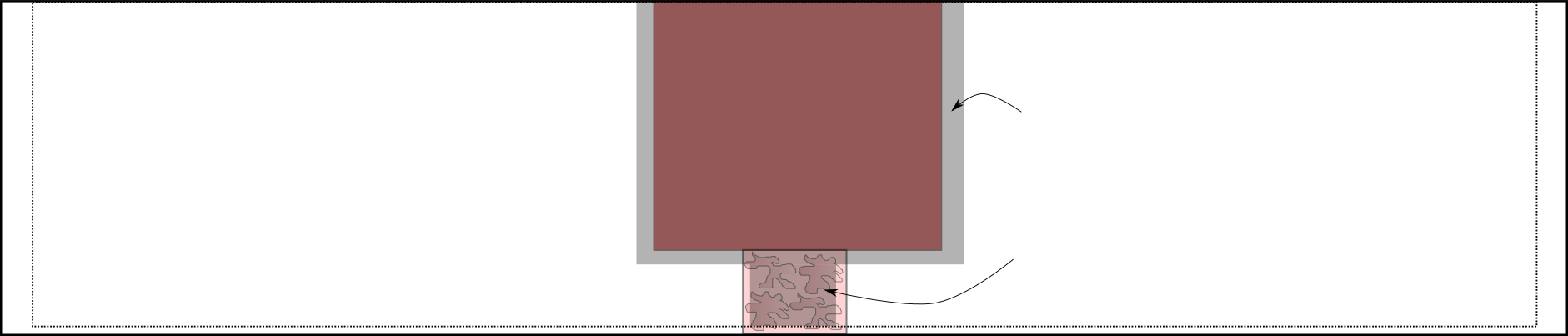}}
\put(-355,59){$\Qbext$}
\put(-322,59){$\Qtext$}
\put(-175,40){\huge $Q_1$}
\put(-114,40){$S_{(1+\iota)}(Q_1)$}
\put(-200,7){$\Qpos$}
\put(-120,20){$S_{(1+\iota)}(\Qzero)$}
\caption{Intermediate cylinders in the proof of the expansion of
  positivity. The great parts are obtained after a scaling with a
  parameter $1+\iota$ close to $1$. This is necessary in order to keep
  $f$ vanishing in a ``good ratio'' of $S_{(1+\iota)}(\Qzero)$.}
  \label{fig:pop-plus}
\end{figure}
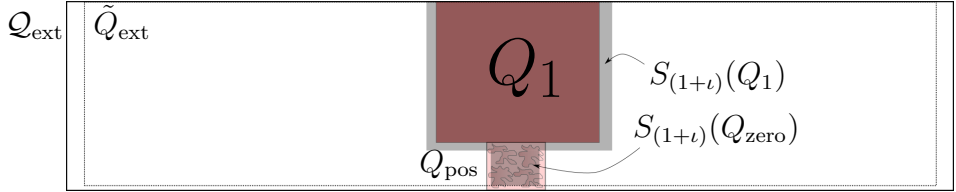

The remainder of the proof is split into several steps. We explain in \textsc{Step 1} how to choose $\iota$ to ensure that cylinders are properly ordered and $\eta$ so that we retain enough information from the assumption \( |\{ f \ge 1 \} \cap \Qpos | \ge \frac12 |\Qpos|\). We then apply the aforementioned successive estimates in \textsc{Step 2}, before deriving the lower bound on $f$ in \textsc{Step~3}. 
\medskip

\noindent \textsc{Step 1.}
Choose $\iota$ small enough so that $\Qext \subset \Qtext \subset S_{(1+\iota)^2}(\Qext)\subset \Qbext$. We only need to check the last inclusion. We choose $\iota >0$ small enough so that $(1+\iota)^4 (1+\eta^2)\leq 1+\theta^2 $, $2(1+\iota)^2\le 3$ and $8(1+\iota)^6\le 9$. Since $\eta\in (0,\frac{\theta}{2})$,  to satisfy the first inequality it is enough to satisfy $(1+\iota)^4 \left(1+\left(\frac{\theta}{2}\right)^2\right)\leq 1+\theta^2 $.  So we pick $\iota = \min \left(\frac{4(1+\theta^2)}{4+\theta^2}-1,\left(\frac{9}{8}\right)^{1/6}-1,\left(\frac{3}{2}\right)^{1/2}-1\right).$

Recall that  $\Qzero = (-1-\eta^2,-1] \times B_{\eta^3} \times B_\eta$. In  particular,  $S_{(1+\iota)}(\Qzero) = (-(1+\iota)^2(1+\eta^2),-(1+\iota)^2] \times  B_{(1+\iota)^3\eta^3} \times B_{(1+\iota) \eta}$.  We next pick    $\eta \in (0,1)$ small enough so that
  \[
    |\Qpos \setminus S_{(1+\iota)} (\Qzero)| \ge \frac14 |S_{(1+\iota)}  (\Qzero)|.
  \]
It is then enough to satisfy $\theta \ge (5/4)^{1/5} (1+\iota) \eta$ so we pick $\eta = (5/4)^{-1/5} (1+\iota)^{-1}\theta $.
 In particular, the previous  volume condition implies that
  \[
    | \{ g =0 \} \cap S_{(1+\iota)}(\Qzero)|\ge |\{f \ge 1 \} \cap S_{(1+\iota)}(\Qzero)| \ge \frac14 |S_{(1+\iota)}(\Qzero)|.
  \]

\noindent \textsc{Step 2.}  
  With such an information at hand, we know that there exists
  $\theta_0 \in (0,1)$, only depending on $\eta$ and thus only
  depending on $\theta$, such that
  \begin{align*}
   \esssup_{Q_1} (g-  \theta_0 G(\eps))_+ 
&   \lesssim \|(g-\theta_0 G(\eps))_+\|_{L^2 (Q_{1+\iota})} +\eta_0  &  \text{from  Proposition~\ref{p:l2linfty}} \\
&   \lesssim_\iota  \|\nabla_v g \|_{L^2(\Qtext)} + (\eta_0/\eps) + \eta_0  & \text{from Theorem~\ref{t:wpi}} \\  
&  \lesssim  \left( \int_{\Qbext} g + 1+ \eta_0 / \eps\right)^{\frac12}+ (\eta_0 /\eps ) + \eta_0   & \text{from Lemma~\ref{l:ee-g}} \\
&  \lesssim  \left( G(\eps) + 2\right)^{\frac12}+2   & \text{for $\eta_0 \le \eps \le 1$} \\
&   \lesssim  \sqrt{G(\eps)} & \text{for $\eps$ such that $G(\eps)\geq 2$}.
\end{align*}
Remark that it has been necessary to scale $g$ before applying
Theorem~\ref{t:wpi}. This generates a constant depending on
$\iota$. This is emphasized by writing $\lesssim_\iota$.  But $\iota$
only depends on dimension, $\lambda$, $\Lambda$ and $\theta$.
\medskip

\noindent \textsc{Step 3.}
The previous computation yields
  \[ g \le C \sqrt{G(\eps)} + \theta_0 G (\eps) \quad \text{ in } Q_1\]
  for some $\theta_0 \in (0,1)$ depending on universal constants and $\theta$.
Since $G(\eps) \to +\infty$ (we can pick $\eps$ and $\eta_0$ small enough only depending on the universal constants and $\theta$) , we thus have 
\[G(f+\eps)=  g \le  \frac{1+2\theta_0}3 G (\eps) \quad \text{ in } Q_1.\]
Now recall that $G(t) \sim - \ln t$ as $t \to 0+$ and $-G'(t)\leq \frac{1}{t}$ for $t\in \left(0,\frac{1}{4}\right]$.  The previous inequality thus
implies that as $\eps \to 0+$,
\[ \ln (f+\eps) \ge \frac{2+\theta_0}3 \ln \eps .\]

This yields the result with $\ell_0 = \eps^{\frac{2+\theta_0}3} - \eps>0$. 
\end{proof}
Before iterating the lemma of expansion of positivity, we state and
prove a straightforward consequence of it that will be used when
applying the Ink-spots theorem. 
\begin{lemma}\label{l:minima-measure}
  Let $m \ge 3$ be an integer and $R$ be given by Lemma \ref{l:pop} for $\theta=m^{-1/2}$. There exists a constant $M>1$ only depending on $m$, $d$, $\lambda$, $\Lambda$ such that  for all non-negative super-solution $f$ \eqref{e:main} with $S=0$ in some cylindrical open set  $\Omega \supset  (-1,m] \times B_{9Rm^{3/2}} \times B_{3Rm^{1/2}}$,  such that
  \[
    |\{ f \ge M \} \cap Q_1 | \ge \frac12 |Q_1|,
  \]
then $f \ge 1$ in $\bar Q_1^m = (0,m] \times B_{m+2} \times B_1$.  
  \end{lemma}
\begin{proof}
  Let $\theta=m^{-\frac{1}{2}}$ so that  $\bar Q_1^m \subset (0,\theta^{-2}] \times B_{\theta^{-3}} \times  B_{\theta^{-1}}$.  We apply Lemma~\ref{l:pop} to $\frac{f}{M}$ with $Q_1$ and $\bar Q_1^m$ taking the role of $\Qpos$ and $Q_1$ thanks to a rescaling argument. This yields that $f \ge \ell_0M$  in $(0,\theta^{-2}] \times B_{\theta^{-3}} \times B_{\theta^{-1}}$.  We
  then pick $M = 1/\ell_0$ and we conclude the proof.
\end{proof}

When deriving the weak Harnack inequality, we will need to estimate
how the lower bound deteriorates with time.  Indeed such an
information is needed in the Ink-spots theorem: since cylinders can
``leak'' out of the set $F$, a corresponding error has to be
estimated, see the term $Cm r_0^2$ in Theorem~\ref{t:ink-spot}.  The
geometric setting is the one from Theorem~\ref{t:weak-harnack}. In
particular, recall that
$Q_+ = (-\omega^2,0] \times B_{\omega^3} \times B_\omega$ and
$Q_- = (-1,-1+\omega^2] \times B_{\omega^3} \times B_\omega$ where
$\omega$ is small and universal. It has to be small enough so that
when spreading positivity from a cylinder $Q_r(z_0)$ from the past,
i.e. included in $Q_-$, the union of the stacked cylinders where positivity is expanded captures $Q_+$.
Then the radius $R_0$ in the statement of weak Harnack inequality is
chosen so that the expansion of positivity lemma can be applied as
long as new cylinders are stacked over previous ones.  These two
facts are stated precisely in the following lemma. 

In order to avoid the situation where the last stacked cylinder
(see $Q[N+1]$ in the next lemma) leaks out of the domain where the
equation is satisfied, we choose it in a way that we can use
the information obtained in the previous cylinder $Q[N]$: the ``predecessor'' of $Q[N+1]$ is contained in $Q[N]$ (see Figure \ref{fig:stackedcylinders}).  
\begin{lemma}[Stacking cylinders]\label{l:cylinder}
  Let $\omega < 10^{-2}$.  Given any non-empty cylinder
  $Q_r (z_0) \subset Q_-$, let $T_k = \sum_{j=1}^k (2^jr)^2$ and
  $N \ge 1$ such that $ T_N \le -t_0 < T_{N+1}$.  Let
\begin{align*}
     Q[k] &=  Q_{2^kr}(z_k)  \text{ for } k =1,\dots, N,\\
  Q [N+1] &= Q_{R_{N+1}}(z_{N+1})
\end{align*}
where $z_k = z_0  \circ (T_k,0,0)$ and, letting $R =|t_0+T_N|^{\frac12}$ and $\rho = (4\omega)^{\frac13}$, $R_{N+1} = \max(R,\rho)$ and 
  \[
  z_{N+1} =
  \begin{cases}
  z_N \circ (R^2, 0,0)  & \text{ if } R \ge \rho, \\
  (0,0,0) & \text{ if } R < \rho.
\end{cases}
\]
These cylinders satisfy
\[
  Q_+ \subset Q[N+1], \qquad 
  \bigcup_{k=1}^{N+1} Q[k]  \subset (-1,0] \times B_{2} \times B_{2}, \qquad   Q[N] \supset \tilde Q[N] 
\]
where $\tilde Q[N]$ is the ``predecessor'' of $Q[N+1]$:
\(\tilde Q[N] = Q_{R_{N+1}/2} (z_{N+1} \circ (-R_{N+1}^2, 0,0)).\)
\end{lemma}


With such a technical lemma at hand, expansion of positivity for large times follows easily. 
\begin{lemma}[Expansion of positivity for large times]\label{l:minima-large-times}
  Let $R_{1/2}$ given by Lemma~\ref{l:pop} with $\theta =1/2$.
  There exist a universal constant $p_0>0$ such that, if $f$ is a
  non-negative weak super-solution of \eqref{e:main} with $S=0$ in
  some cylindrical open set $\Omega \supset Q=(-1,0] \times B_{18 R_{1/2}} \times B_{6 R_{1/2}}$ such
  that
  \[
    |\{ f \ge A \} \cap Q_r (z_0)| \ge \frac12 |Q_r(z_0)|
  \]
  for some $A>0$ and for some cylinder $Q_r(z_0) \subset Q_-$, then
  $f \ge A (r^2/4)^{p_0}$ in $Q_+$. 
\end{lemma}
\begin{proof}
We first apply Lemma~\ref{l:pop} with $\theta=1/2$
to $f/A$ (after rescaling $Q_r(z_0)$ into $\Qpos$) and get
$f/A \ge \ell_0$ in $Q[1]$. We then apply it to $f / (A \ell_0)$ and
get $f \ge A \ell_0^2$ in $Q[2]$. By induction, we get
$f \ge A \ell_0^k$ in $Q[k]$ for $k=1,\dots, N$.

We then apply Lemma~\ref{l:pop} one more time and get
$f \ge A \ell_0^{N+1}$ in $Q[N+1]$ and in particular
$f \ge A \ell_0^{N+1}$ in $Q_+$. Since $T_N \le 1$, we have
$4^N r^2 \le 1$. Choosing $p_0>0$ such that $\ell_0 = (1/4)^{p_0}$, we
get $f \ge A ((1/4)^{N+1})^{p_0} \ge A (r^2/4)^{p_0}$.
\end{proof}

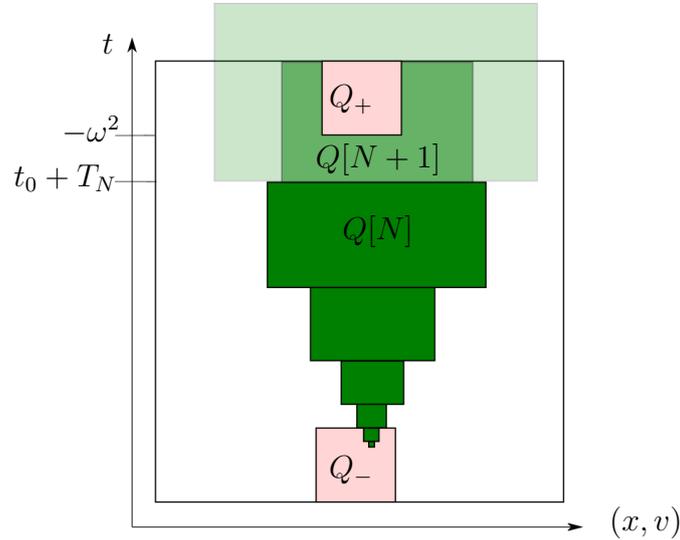
\begin{figure}[h]
\centering{\includegraphics[height=7cm]{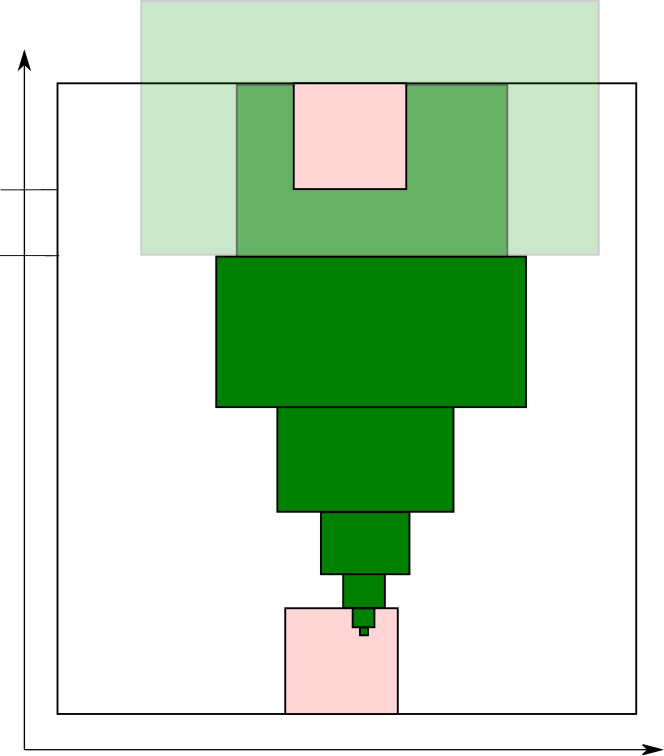}}
\put(-195,147){$-\omega^2$}
\put(-213,130){$t_0+T_N$}
\put(10,0){$(x,v)$}
\put(-180,180){$t$}
\put(-95,20){$Q_-$}
\put(-95,160){$Q_+$}
\put(-90,110){$Q[N]$}
\put(-100,137){$Q[N+1]$}
\caption{{\it Stacking cylinders above an initial one contained in
    $Q_-$.} We see that the stacked cylinder obtained after $N+1$
  iterations by doubling the radius leaks out of the domain. This is
  the reason why $Q[N+1]$ is chosen in a way that it is contained in
  the domain and its ``predecessor'' is contained in $Q[N]$. Notice
  that the cylinders $Q[k]$ are in fact slanted since they are not
  centered at the origin. We also mention that $Q[N+1]$ is choosen
  centered if the time $t_0+T_N$ is too close to the final time $0$.}
  \label{fig:stackedcylinders}
\end{figure}

We finally turn to the proof of the main result of this paper,
Theorem~\ref{t:weak-harnack}.
\begin{proof}[Proof of Theorem~\ref{t:weak-harnack}]
  We start the proof with general comments about the geometric
  setting. The proof is going to use a covering argument through the
  application of the Ink-spots theorem. To apply this result, we will
  consider an arbitrary cylinder $Q$ contained in $Q_-$. The
    parameter $\omega$ used in the definition of $Q_-$ and $Q_+$ is
    chosen small enough ($\omega\leq 10^{-2}$) so that the cylinder
    $Q_+$ is ``captured'' when stacking cylinders
    (Lemma~\ref{l:cylinder}) and propagating positivity
    (Lemma~\ref{l:minima-large-times}). We also pick the parameter
    $R_0$ in the definition of the cylinder $Q^0$ large enough so that
    the stacked cylinders do not leak out of $Q^0$; we impose
    $R_0 \ge 18 R_{1/2}$ where $R_{1/2}$ is given by Lemma~\ref{l:pop}
    for $\theta =1/2$. We also impose
    $R_0 \ge 9 R_{m^{-1/2}} m^{3/2} \omega^3$ where $R_{m^{-1/2}}$ is
    given by Lemma~\ref{l:pop} with $\theta = m^{-1/2}$ in order to be
    in position to apply Lemma~\ref{l:minima-measure} to cylinders
    contained in $Q^-$, hence of radius smaller than $\omega$. 
  \medskip

  We first classically reduce to the case
  \[ \inf_{Q_+} f \le 1 \quad \text{ and } \quad S=0. \] Considering
  $\tilde{f}(t,x,v)=f(t,x,v)+\|S \|_{L^{\infty}} t$,
  $\tilde{f}$ is a super-solution of the same equation with no source
  term ($S=0$) and the weak Harnack inequality for $\tilde{f}$ implies
  the one for $f$.  So from now we assume $S=0$.  Considering next
  $\tilde{f} = f / (\inf_{Q_+} f +1)$ reduces to the case $\inf_{Q_+} f \le 1$.

  We then aim at proving that $\int_{Q_-} f^p (z) \dz$ is bounded from
  above by a universal constant for some universal exponent $p$. 
  This  amounts to prove that for all $k \in \mathbb{N}$,
  \[ |\{ f > M^k\} \cap Q_- | \le \Cwhi (1-\tilde \mu)^k \] for some
  universal parameters $\tilde\mu \in (0,1)$, $M>1$ and $\Cwhi>0$ to be
  determined later. We can see that this property would be enough by
  transposing it to the continuous case ($k$ real and above $1$) and
  by application of the layer cake formula to $\int_{Q_-} f^p (z) \dz$.

  We are going to apply Theorem~\ref{t:ink-spot} with $\mu =1/2$. We pick
  $m \in \mathbb{N}$ such that $\frac{m+1}m (1-c/2) \le 1 -c/4$.  Then
  the constant $M>1$ is given by Lemma~\ref{l:minima-measure}.
  \medskip
  
  We prove the result by induction. For $k=1$, we simply choose
  $\tilde\mu \le 1/2$  and $\Cwhi$ such that
  $|Q_-| \le \frac{1}{2}\Cwhi$. Now assume that the claim holds true for
  $k \ge 1$ and let us prove it for $k+1$.  We thus consider
  \[ E = \{ f > M^{k+1} \} \cap Q_- \quad \text{ and } \quad F = \{ f > M^k \} \cap Q_1. \]
  These two sets are bounded and measurable and such that $E \subset F \cap Q_-$.
  We consider a cylinder $Q =Q_r (t,x,v) \subset Q_-$ such that $|Q \cap E | > \frac12 |Q|$, that is to say
  \[ |\{f > M^{k+1} \} \cap Q | > \frac12 |Q|.\]

  We first prove that $r$ is small, i.e. we determine a universal
  $r_0$ which depends on $k$ such that $r < r_0$.
  Lemma~\ref{l:minima-large-times} (after translation in time) implies
  that \( f \ge M^{k+1} (r^2/4)^{p_0} \) in $Q_{+}$. In particular,
  $1 \ge \inf_{Q_+} f \ge M^{k+1} (r^2/4)^{p_0}$ so
  $r^{2p_0}\leq 4^{p_0} M^{-(k+1)}$.  We thus choose $r_0 = 2 M^{-\frac{k+1}{2p_0}}$.

  We next prove that $\bar Q^m \subset F$, i.e.
  \[ \bar Q^m \subset \{ f > M^k\} .\] In order to do so, we 
  apply Lemma~\ref{l:minima-measure} to $f/M^k$ after rescaling $Q$ in
  $Q_1$ where we assume $\omega \leq (2m+3)^{-1/2}$ to be able to rescale.

  By Theorem~\ref{t:ink-spot}, we conclude thanks to the induction assumption that
  \begin{align*}
    |\{ f > M^{k+1} \} \cap Q_-| &\le  (1-c/4) \bigg(\Cwhi (1 -\mu)^k +  C m r_0^2 \bigg)  \\
                                 & \le  (1-c/4) \bigg(\Cwhi (1 -\mu)^k +  C m M^{-\frac{k+1}{p_0}} \bigg).  \\
    \intertext{Then pick $\tilde \mu$ small enough so that $M^{-1/p_0} \le (1- \mu)$ and $\tilde\mu \le \frac{c}{4}$ and get,}
                                 & \le \Cwhi (1-c/4) \bigg(1  + \Cwhi^{-1} C m M^{-\frac{1}{p_0}} \bigg) (1 -\tilde\mu)^k.  \\
    \intertext{Now pick $\Cwhi$ large enough (depending on $c$, $C$, $m$ and $M^{-1/p_0}$) and get,}
    & \le \Cwhi (1 - c/4) (1 -\tilde\mu)^k\\
     & \le \Cwhi (1 -\tilde\mu)^{k+1}. 
  \end{align*}
  The proof is now complete.
\end{proof}
The full Harnack inequality is a direct consequence of the local
boundedness of sub-solutions and the weak Harnack inequality.
\begin{proof}[Proof of Theorem~\ref{t:harnack}]
 Combine Proposition~\ref{p:l2linfty} and
  Theorem~\ref{t:weak-harnack} and rescale to reach the result.  See for example \cite{LZ17} for more details. 
\end{proof}

\appendix

\section{Appendix: the Ink-spots theorem}
\label{a:ink}

In order to state the Ink-spots theorem, we need to define stacked
cylinders.  Given $Q = Q_r (t,x,v)$ and $m \in \mathbb{N}$, $\bar Q^m$
denotes the cylinder
$\{ (t,x,v): 0<t-t_0\le m r^2, |x-x_0 - (t-t_0)v_0| < (m+2) r^{3},
|v-v_0| < r \}$.

\begin{theorem}[Ink-spots -- \cite{is}] \label{t:ink-spot} Let $E$ and
  $F$ be two bounded measurable sets of $\R \times \R^{2d}$ with
  $E \subset F \cap Q_-$. We assume that there exist two constants
  $\mu, r_0 \in (0,1)$ and an integer $m \in \mathbb{N}$ such that for
  any cylinder $Q \subset Q_-$ of the form $Q_r (z_0)$ such that
  $|Q \cap E | \ge (1-\mu)|Q|$, we have $\bar Q^m \subset F$ and
  $r < r_0$.  Then
  \[ |E| \le \frac{m+1}m (1-c\mu) \bigg( |F \cap Q_-| + C m r_0^2 \bigg)\]
  where $c \in (0,1)$ and $ C>0$ only depend on dimension $d$. 
\end{theorem}

\begin{remark}
  This corresponds to \cite[Corollary~10.1]{is} with $Q_-$ instead of $Q_1$, i.e. the Ink-spots
  theorem with leakage, with $s=1$.  Indeed, the statement in
  \cite{is} is more general since the cylinders are of the form
  $z_0 \circ Q_r$ with
  $Q_r = (-r^{2s},0] \times B_{r^{1+2s}} \times B_r$ for some
  $s \in (0,1]$. In the statement above, we only deal with $s=1$.
\end{remark}

\section{Appendix : local H\"older estimate}
\label{a:holder}

The H\"older estimate from Theorem~\ref{t:holder} is classically
obtained by proving that the oscillation of the solution decays by a
universal factor when zooming in. Such an improvement of oscillation
is obtained from Lemma~\ref{l:pop} with $\theta =1$.

Of course, it is not necessary to prove this lemma in order to prove
the Harnack inequality since the H\"older estimate can be derived from
it. Eventhough, we provide a proof to emphasize that it can be easily
derived from Lemma~\ref{l:pop}. 
\begin{lemma}[Decrease of oscillation]\label{l:osc}
Let $\bar R >0$ be such that $Q_{\bar R} \supset (-2,0] \times B_{9 R_1} \times B_{3R_1}$ with $R_1$ universal given by Lemma~\ref{l:pop} with $\theta =1$.
  There exist (small) universal constants $\eta_0,\ell_0>0$ such that for  any solution $f$ of \eqref{e:main} in some cylindrical open set  $\Omega \supset Q_{\bar R}$
  such that $0 \le f \le 2$ in $Q_{\bar R}$ and  $\|S\|_{L^\infty (Q_{\bar R})} \le \eta_0$, then  $\osc_{Q_1} f \le 2 -\ell_0$.
\end{lemma}
\begin{proof}
  Remark that either
  $| \{f \le 1 \} \cap (-2,-1] \times B_1 \times
  B_1 | \le \frac12 |(-2,-1] \times B_1 \times
  B_1 |$
  or
  $| \{f \le 1 \} \cap (-2,-1] \times B_1 \times
  B_1 | \ge \frac12 |(-2,-1] \times B_1 \times
  B_1 |$.
  In the former case, Lemma~\ref{l:pop} implies that $f \ge \ell_0$
  in $Q_1$ while in the latter, we simply consider $\tilde f = 2 -f$,
  apply Lemma~\ref{l:pop} to this new function and get
  $f \le 2 -\ell_0$ in $Q_1$. In both cases, we get the desired reduction
  of oscillation: $\osc_{Q_1} f \le 2 - \ell_0$.
\end{proof}

Deriving Theorem~\ref{t:holder} from Lemma~\ref{l:osc} is completely
standard but we provide details for the sake of completeness and for
the reader's convenience.
\begin{proof}[Proof of Theorem~\ref{t:holder}]
  Let $f$ be a solution of \eqref{e:main} in $Q_1$. By scaling, we can
  reduce to the case $\|f\|_{L^2(Q_{\bar R})} \le 1$ and
  $\|S\|_{L^\infty (Q_{\bar R })} \le \eta_0$ where $\eta_0$ is given by
  Lemma~\ref{l:osc}. We deduce from
  Proposition~\ref{p:l2linfty} that $f$ is bounded in $Q_{1}$.

 In order to prove that $f$ is H\"older continuous in
 $Q_{1/2}$, it is sufficient to prove that for all
 $z_0 \in Q_{1/2}$ and $r \in (0,1/(9R_1))$,
 \[ \osc_{Q_r (z_0)} f \le C_\alpha r^\alpha \] for some universal
 constants $\alpha \in (0,1)$ and   $C_\alpha$.

 We reduce to the case $z_0=0$ by using the invariance of the
 equation by the transformation $z \mapsto z_0 \circ z$ and we simply prove
 \[ \osc_{Q_{\bar R^{-k}}} f \le C (1-\delta_0)^k \] for some $C$ and
 $\delta_0 = \ell_0/2 \in (0,1)$ universal. By scaling, this amounts to
 prove that if $\osc_{Q_{\bar R}} f \le 2$   then
 $\osc_{Q_1} f \le 2( 1-\delta_0)$.  By considering
 $\tilde f = 1+ \frac{f}{\|f\|_{L^\infty(Q_{\bar R})} + \|S\|_{L\infty
     (Q_{\bar R})}/\eta_0}$, we can assume that $0 \le f \le 2$ and
 $|S| \le \eta_0$ in $Q_{\bar R}$. Remark that the $L^\infty$ bound of $S$ is
 reduced when zooming in. We now apply Lemma~\ref{l:osc} and conclude.
\end{proof}

\section{Appendix: stacking cylinders}
\label{a:cylinder}

\begin{proof}[Proof of Lemma~\ref{l:cylinder}]
We first check that the sequence of cylinders is well defined for $\omega <1/\sqrt5$, say. Since
$r \le \omega$, we have $t_0+T_1 \le -1 +\omega^2 + 4r^2 < 0$
and we know that there exists $N \ge 1$ such that
$T_N< -t_0 \le T_{N+1}$. 

We check next that $Q_+ \subset Q[N+1]$. If $R < \rho$, we simply remark that $\omega \le \rho$ to conclude.
In the other case, when $R \ge \rho$, we have to prove that $Q_\omega (z_{N+1}^{-1}) \subset Q_R$. In this case,
we have $z_{N+1}^{-1} = (0,x_0 -t_0 v_0,v_0)^{-1} = (0, -x_0 + t_0 v_0, -v_0)$ and for $z \in Q_\omega$,
\[ z_{N+1}^{-1} \circ z  = (t,-x_0+t_0 v_0 + x -tv_0, v-v_0) \in Q_R\]
if $\omega^2 \le R^2$ and $\omega^3 + \omega + \omega^3 + \omega^3 \le R^3$ and $2 \omega \le R$.
This is true for $4\omega \le R^3$ that is to say $\rho \le R$. 
\medskip

Let us now check that for all $k \in \{1,\dots, N+1\},$
$Q[k] \subset (-1,0] \times B_{2} \times B_2$.

As far as $Q[N+1]$
is concerned, we use the fact that $R = |t_0+T_N|^{\frac12} \le 1$ and
$\rho = (4\omega)^{\frac13} \le 1$ to get $R_{N+1} \le 1$. Moreover
$z_{N+1} \in Q_1$ and thus $Q[N+1] \subset (-1,0]\times B_2 \times B_2$.

We remark
$T_N \le -t_0 \le 1$ implies that $(2^Nr)^2 \le \frac34 +r^2 \le 1$ 
and in particular $2^N r \le 1$.  If $\bar z_k =(t_k,x_k,v_k) \in Q[k]$ for
$k \le N$ then there exists $(t,x,v) \in Q_1$ such that
$\bar z_k = z_0 \circ (T_k,0,0) \circ ((2^kr)^2t,(2^k r)^3x,2^k r v)$. This
implies that $x_k = x_0 +T_k v_0+ (2^kr)^2 t v_0 + (2^kr)^3 x$ and
$v_k = v_0 + 2^k r v$ and since $z_0 \in Q_-$,
\[ |x_k| \le \omega^3 +  2\omega + 1 \le 2 \text{ and } |v_k| \le \omega + 1 \le 2.\]
In particular $Q[k] \subset (-1,0] \times B_2 \times B_2$. 
\medskip

We are left with proving that $\tilde{Q}[N]  \subset Q[N]$.

If $R \ge \rho$, then the conclusion follows from the fact that $R/2 \le 2^N r$ (since $T_{N+1} >0$). 

Let us deal with the case $R \le \rho$. In view of the definitions of these cylinders, this is equivalent to
\[Q_{\rho/2} (\bar z) \subset Q_{2^N r} \text{ with } \bar z = (-T_N,0,0) \circ z_0^{-1} \circ (-\rho^2,0,0).\]

In order to prove this inclusion, we first estimate $2^N r$ from below. Since $t_0+T_{N+1} > 0$ and $-t_0 \ge 1 -\omega^2$, we have
$(4/3)(4^{N+1}-1)r^2 \ge 1-\omega^2$ and in particular $4^N r^2 \ge (1/4) (3/4 - 7/4 \omega^2) \ge 1/8$. We conclude
that
\begin{equation}\label{e:lower}
  2^N r \ge 1/(2\sqrt2).
\end{equation}

With such a lower bound at hand, we now compute $\bar z = (R^2-\rho^2, -x_0 + (t_0+\rho^2)v_0, -v_0)$ and get
for $z \in Q_{\rho/2}$,
\[ \bar z \circ z = (R^2-\rho^2 + t,  -x_0 + (t_0+\rho^2)v_0 + x -tv_0,v-v_0) \in Q_{2\rho}. \]
Indeed, $-2\rho^2 < R^2 -\rho^2 + t \le 0$ and $|-x_0 + (t_0+\rho^2)v_0 + x -tv_0| \le \omega^3 +3 \omega + (\rho/2)^3 \le 2 \rho^3$ and $|v-v_0|\le 2 \rho$.
It is thus sufficient to pick $\omega$ such that $\rho \le 1/(2\sqrt2)$ to get the desired inclusion.
This is true for $\omega \le 10^{-2}$. 
\end{proof}

\bibliographystyle{siam}
\bibliography{kk}

\end{document}